\documentclass[reqno]{amsart}

\usepackage[absolute,overlay]{textpos}
\usepackage{eso-pic}
\usepackage[a4paper,margin=1.5in]{geometry}
\usepackage{mathrsfs}
\usepackage{amsfonts,  amsmath}
\usepackage{graphicx}
\usepackage{enumerate,  enumitem,  amscd}
\usepackage{amsthm,  amssymb,  epsfig,  hyperref,  amsmath}  
\usepackage{mathtools}
\usepackage{tikz}
\usepackage{tikz-cd}
\usetikzlibrary{arrows,automata,arrows.meta}
\usepackage{quiver}
\usepackage{subcaption}

\tikzset{
  vx/.style  = {circle, draw=black, fill=blue!20, line width=0.6pt,
                minimum size=8pt, inner sep=0pt},
  ed/.style  = {draw=black!60, line width=0.5pt},
  lbl/.style = {font=\tiny},
  ttl/.style = {font=\small}
}

\newcommand{\Z}{\mathbb{Z}}

\newcommand{\N}{\mathbb{N}}

\newcommand{\cR}{\mathcal{R}}

\newcommand{\eps}{\varepsilon}
\newcommand{\FF}{\mathrm{F}}
\newcommand{\ZM}{\mathbb{Z}M}
\newcommand{\cE}{\mathcal{E}}
\newcommand{\cQ}{\mathcal{Q}}
\newcommand{\cZ}{\mathcal{Z}}

\DeclareMathOperator{\FIM}{FIM}
\DeclareMathOperator{\FP}{FP}
\DeclareMathOperator{\id}{id}

\newtheorem{theorem}{Theorem} 
\newtheorem*{theorem*}{Theorem} 
\numberwithin{theorem}{section}
\newtheorem{lemma}[theorem]{Lemma}     
\newtheorem{corollary}[theorem]{Corollary}
\newtheorem*{corollary*}{Corollary}

\newtheorem*{proposition*}{Proposition}
\numberwithin{equation}{section}

\theoremstyle{definition}

\numberwithin{example}{section}
\newtheorem*{question*}{Question}
\numberwithin{question}{section}

\newtheorem*{remark*}{Remark}

\begin{document}

\title[On homological finiteness and free inverse monoids]{On homological finiteness properties and free inverse monoids}
\author{Carl-Fredrik Nyberg-Brodda}
\address{June E Huh Center for Mathematical Challenges, Korea Institute for Advanced Study (KIAS), Seoul 02455, Korea}
\email{cfnb@kias.re.kr}

\thanks{The author is currently supported by KIAS Individual Grant HP094701 at Korea Institute for Advanced Study, and by the Mid-Career Researcher Program (RS-2023-00278510) through the National Research Foundation funded by the government of Korea.}

\date{\today}


\begin{abstract} 
We construct a simple and useful sufficient condition, based on actions on a lattice of idempotents, for monoids admitting homomorphisms to the monogenic free inverse monoid $\FIM(1)$ to not be of type $\FP_2$. This recovers a result of Gray and Steinberg that free inverse monoids are not of type $\FP_2$. The same technique is then used to show that a finitely generated submonoid of $\FIM(1)$ is of type $\FP_2$ if and only if it is finitely presented, answering a question of Cho \& Ru\v{s}kuc. 
\end{abstract}

\maketitle 


\noindent It is by now a well-established fact that finite presentability of inverse monoids qua monoids is a rather delicate issue. One of the most fundamental results in this direction is due to Schein \cite{Schein1975}, who proved that \textit{no} non-trivial rank $r$ free inverse monoid $\FIM(r)$ -- not even when $r=1$, i.e.\ the monogenic one -- is finitely presented as a monoid. Recently, Cho \& Ru\v{s}kuc \cite{ChoRuskuc2025} classified which finitely generated submonoids of the monogenic free inverse monoid are finitely presented as monoids. On the other hand, there is a strong connection between finite presentability and \textit{homological} finiteness properties, particularly the finiteness property $\FP_2$. Indeed, it is easy to show that any finitely presented group (or monoid) satisfies the finiteness property $\FP_2$, but e.g.\ it is a celebrated result of Bestvina \& Brady \cite{Bestvina1997} that $\FP_2$ does not imply finite presentability for groups.

In line with this, there arises a natural program of proving known non-finitely presented submonoids to be not of type $\FP_2$. For example, the author \cite{NybergBrodda2025} has recently shown that free regular $\star$-monoids, an object closely related to inverse monoids, are not of type $\FP_2$. Importantly, Gray \& Steinberg \cite{GraySteinberg2021} proved, by a topological argument via normal forms, that free inverse monoids are not of type $\FP_2$. Cho \& Ru\v{s}kuc \cite{ChoRuskuc2025} thus asked whether one can classify which submonoids of $\FIM(1)$ are of type $\FP_2$. To answer this question, we develop a sufficient and geometric condition, based on grid actions and idempotents, for a monoid to not be of type $\FP_2$ (Theorem~\ref{Thm:main-theorem}).  This lets us prove that any finitely generated inverse monoid which admits a surjection onto the monogenic free inverse monoid $\FIM(1)$ is necessarily not of type $\FP_2$ (Corollary~\ref{Cor:if-surjects-then-not-FP2}). In particular, we recover in a new and geometric way the aforementioned result of Gray \& Steinberg. We also obtain an answer to the question of Cho \& Ru\v{s}kuc (Corollary~\ref{Cor:cho-ruskuc-question}).

I would like to thank Ben Steinberg for helpful comments and encouragement.

\section{Background and setup}

\noindent We assume the reader is familiar with the basics of inverse semigroup theory, and if they are not, then we direct the reader to e.g.\ the standard and excellent book by Lawson \cite{Lawson1998}. We shall also assume the reader is comfortable with the elements of homological algebra, to which we refer the reader e.g.\ to Rotman \cite{Rotman1979}. 

\subsection{Homological setup}\label{Subsec:hom-setup} Let $M$ be a finitely generated monoid, and let $\ZM$ be its monoid ring. We will consider resolutions of the trivial right $\ZM$-module $\Z$. In particular, if $\Z$ admits a resolution 
\begin{equation*}\label{Eq:resolution}
\cdots \to P_n \to P_{n-1} \to \cdots \to P_1 \to P_0 \to \Z \to 0,
\end{equation*}
which we will abbreviate as $P_\bullet \twoheadrightarrow \Z$, such that $P_0, \dots, P_n$ are finitely generated projective $\ZM$-modules, then we say that $M$ is of type \textit{right-$\FP_n$}. If $M$ is of type right-$\FP_n$ for all $n \geq 0$, then we say it is of type right-$\FP_\infty$. For inverse monoids, the inversion involution $\circ^{-1}$ converts any right $\ZM$-module into a left one in the obvious way, and hence left- and right-$\FP_n$ are equivalent; we shall therefore refer to this property as $\FP_n$ when speaking of inverse monoids. Any finitely generated monoid is of type (left- and right-)$\FP_1$, and any finitely presented monoid is of type (left- and right-)$\FP_2$. Both converses fail in general, see \cite{Kobayashi2007}; indeed any monoid with a zero is automatically of type (left- and right-)$\FP_\infty$, see \cite{Cohen1992, Kobayashi2007}. 

There is a standard resolution for any monoid $M$ generated by a finite set $A$, taking $P_0 = \ZM$ and $P_1 = (\ZM)^{A}$, and proceeding to resolve the syzygy module $K_A := \ker(P_1 \to P_0)$. The chain map $\partial_1 \colon P_1 \to P_0$ is defined on the basis element $(a)$, corresponding to $a \in A$, by mapping $(a) \mapsto a-1$. We recall here the basic but very useful fact that the generalized Schanuel's Lemma implies that $M$ is of type right-$\FP_2$ if and only if the syzygy module $K_A$ is finitely generated as a right $\ZM$-module. This will be exploited in the sequel, i.e.\ our manner of proving that $\FP_2$ fails in the monoids under consideration is by showing that $K_A$ is not finitely generated. 

\subsection{The grid complex}\label{Subsec:the-grid-complex} Our argument will only be topological in a very narrow sense, in that we will rely on an action of monoids on certain grids of idempotents. We will only require one topological space: let $\cQ$ be the square CW complex with vertices $\N^2$, unit horizontal and vertical edges, and all unit squares filled by a $2$-cell, so that its geometric realization is all of $[0, \infty)^2$. In coordinates, we let $h_{i,j}$ be the horizontal edge $(i, j) \to (i+1, j)$, and likewise $v_{i,j}$ the vertical edge $(i, j) \to (i, j+1)$. Finally, let $q_{i,j}$ be the $2$-cell glued to $h_{i,j}, v_{i+1,j}, -h_{i,j+1}, -v_{i,j}$ in order. Note that $\cQ$ is obviously contractible. It follows that the augmented cellular chain complex 
\begin{equation}\label{Eq:chain-complex-of-Q}
0 \to C_2(\cQ) \xrightarrow{\partial_2} C_1(\cQ) \xrightarrow{\partial_1} C_0(\cQ) \xrightarrow{\varepsilon} \Z \to 0 
\end{equation}
is exact, where $\partial_1, \partial_2$ are the obvious attachment maps above, and hence that $\partial_2$ induces an isomorphism $\partial_2 \colon C_2(\cQ) \xrightarrow{\sim} \cZ_1(\cQ)$, where $\cZ_1(\cQ) \subseteq C_1(\cQ)$ denotes the $1$-cycles. 

\subsection{The free inverse monoid and incomparable idempotents}\label{Subsec:fim1} Let $\FF_1 = \FIM(1)$ be the monogenic free inverse monoid. We can, as for all free inverse monoids, model its elements as Munn trees, which in the monogenic case simply becomes a pair $(I, t)$, where $I \subseteq \Z$ is a finite non-empty interval with $0 \in I$, and $t \in I$ is a distinguished terminal vertex. We will call an interval containing $0$ a \textit{Munn interval}, in view of this. 

An element $(I, t)$ of $\FIM(1)$ is idempotent if and only if $t=0$. The lattice of idempotents $E(\FF_1)$ of $\FF_1$ is denoted $\cE_1$. Every Munn interval $I$ is of the form $[-i, j]$ for some $i, j \in \N$, and we will denote the idempotent $([-i, j], 0)$ by $\eps_{i,j} \in \cE_1$. Then it is easy to see that the map $\eps_{i,j} \mapsto (i, j) \in \N^2$ is an isomorphism $\cE_1 \cong \N^2$ of semilattices when $\N^2$ is taken to be the usual semilattice with the max operation in each coordinate. The \textit{level} $\lambda(e)$ of an idempotent $e = (I, 0)$ is defined as the edge-length of the interval $I$. Let $e, f \in \cE_1$ be two idempotents. We say that they are \textit{incomparable} if $ef \not\in \{ e, f\}$. Equivalently, if we write $e=\varepsilon_{i,j}$ and $f = \varepsilon_{k, \ell}$, then $e$ and $f$ are incomparable if and only if the interval $[-i, j]$ is not contained in $[-k, \ell]$, and vice versa. The level $\lambda(ef)$ of the product of two incomparable idempotents is then the edge-length of the union of the intervals, i.e.\ $\max(i, k) + \max(j, \ell)$.

\section{A sufficient condition for non-$\FP_2$}

\subsection{The grid action}\label{Subsec:grid-action}

Let $\cQ$ be defined as in \S\ref{Subsec:the-grid-complex}. Recall the canonical right action of $\FF_1 = \FIM(1)$ on its idempotent lattice $\cE_1 = E(\FF_1)$ by $e \circ s = s^{-1}e s$, for $s \in \FF_1$ and $e \in \cE_1$. Let $x$ be a generator of $\FF_1$. Then under the isomorphism $\cE_1 \cong \N^2$ of semilattices, geometrically the action of $x$ on $\cE_1$ amounts to shifting all idempotents one step south-east along a diagonal, and sliding one step to the right on the horizontal axis; while $x^{-1}$ shifts one step north-west and slides up on the  vertical axis. This action is shown in Figure~\ref{Fig:action}. Thus we have an action on $\cQ^{(0)} = \N^2$. We now show that this extends to an action of $\FF_1$ on the entire complex $\cQ$, by mapping an edge $a \to b$ to the edge $a \cdot s \to b \cdot s$ if the endpoints remain distinct, and otherwise map it to a point. 

\begin{figure}[ht]
\centering

\begin{subfigure}[t]{0.47\textwidth}
\centering
\begin{tikzpicture}[
    scale=0.9,
    vertex/.style={circle,fill=black,inner sep=1.1pt},
    maparrow/.style={-{Latex[length=2.8mm,width=2mm]},blue!75!black,line width=1.15pt},
    levelline/.style={gray!55,dashed,line width=0.5pt},
    leveltext/.style={gray!70,font=\scriptsize},
    lab/.style={font=\small}
  ]

  \draw[->] (-0.35,0) -- (4.45,0) node[right] {$i$};
  \draw[->] (0,-0.35) -- (0,4.45) node[above] {$j$};

  \foreach \i in {0,...,4}
    \foreach \j in {0,...,4}
      \node[vertex] at (\i,\j) {};

  \foreach \i in {0,...,3}
    \foreach \j in {1,...,4}
      \draw[maparrow] (\i,\j) -- (\i+1,\j-1);

  \foreach \i in {0,...,3}
      \draw[maparrow] (\i,0) -- (\i+1,0);

  \node[lab] at (2,-0.82) {$(i,j)\cdot x=(i+1,\max\{j-1,0\})$};

\end{tikzpicture}
\caption{The action of $x$.}
\end{subfigure}
\hfill
\begin{subfigure}[t]{0.47\textwidth}
\centering
\begin{tikzpicture}[
    scale=0.9,
    vertex/.style={circle,fill=black,inner sep=1.1pt},
    maparrow/.style={-{Latex[length=2.8mm,width=2mm]},red!75!black,line width=1.15pt},
    levelline/.style={gray!55,dashed,line width=0.5pt},
    leveltext/.style={gray!70,font=\scriptsize},
    lab/.style={font=\small}
  ]

  \draw[->] (-0.35,0) -- (4.45,0) node[right] {$i$};
  \draw[->] (0,-0.35) -- (0,4.45) node[above] {$j$};

  \foreach \i in {1,...,4}
    \foreach \j in {0,...,3}
      \draw[maparrow] (\i,\j) -- (\i-1,\j+1);

  \foreach \j in {0,...,3}
      \draw[maparrow] (0,\j) -- (0,\j+1);
      
  \foreach \i in {0,...,4}
    \foreach \j in {0,...,4}
      \node[vertex] at (\i,\j) {};

  \node[lab] at (2,-0.82) {$(i,j) \cdot x^{-1} =(\max\{i-1,0\},j+1)$};

\end{tikzpicture}
\caption{The action of $x^{-1}$.}
\end{subfigure}

\caption{The conjugation action of the generators of $\FF_1 = \langle x, x^{-1}\rangle$ on the idempotents of $\FF_1$, modelled as an action on $\N^2 = \cQ^{(0)}$.}
\label{Fig:action}
\end{figure}
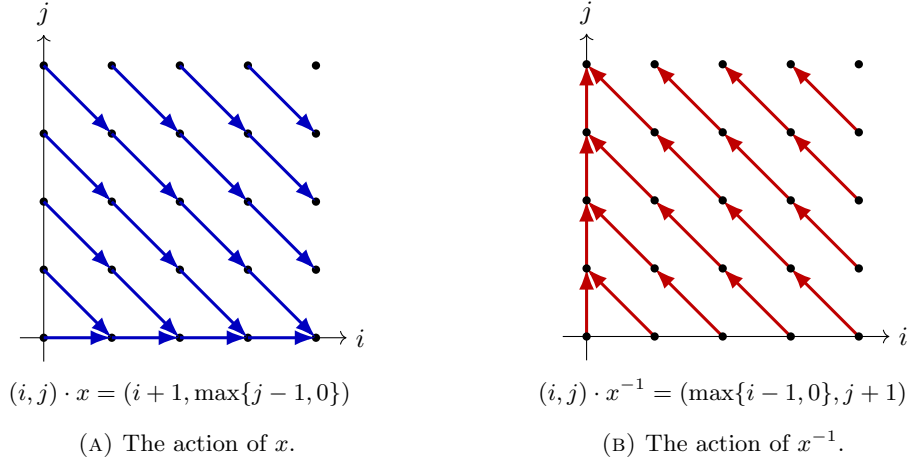

We can easily describe this in coordinates. If $s = (I, t) \in \FF_1$ is an element, and $I = [-a, b]$ with $a, b \in \N$, then its aforementioned action on points $(i, j) \in \cQ^{(0)}$ becomes 
\[
(i, j) \cdot s = ( \max(i,a)+t , \: \max(j,b) - t) 
\]
while the action on horizontal resp.\ vertical edges becomes
\begin{align*}
h_{i,j} \cdot s =  \begin{cases*} h_{i+t, \max(j,b)-t} & if $i\geq a$, \\ 0 & otherwise. \end{cases*} \:\: \text{and} \:\:
v_{i,j} \cdot s = \begin{cases*} v_{\max(i,a)+t, j-t} & if $j\geq b$, \\ 0 & otherwise. \end{cases*}
\end{align*}
Finally, the action on the $2$-cell $q_{i,j}$ is the natural one extending the edge action, i.e.\ 
\begin{align*}
q_{i,j} \cdot s = \begin{cases*} q_{i+t, j-t} & if $i\geq a$ and $j\geq b$, \\ 0 & otherwise. \end{cases*}
\end{align*}
We claim that this is, in fact, an action on $\cQ$, i.e.\ that $C_1(\cQ)$ and $C_2(\cQ)$ are right $\Z\FF_1$-modules and that $\partial_k(c \cdot s) = \partial_k(c) \cdot s$ for all $s \in \FF_1$ and $c \in C_k(\cQ)$ where $k = 1, 2$. Indeed, note that the action on the vertices $(i,j)$ is clearly associative, since as discussed above it is just the action of $\FF_1$ on its idempotent lattice by conjugation. Next, the images of two adjacent vertices under that action clearly are either adjacent or equal, and the action on the horizontal $h_{i,j}$ described above sends this edge either to the edge making its endpoints adjacent, in the first case, or else to $0$, in the second case. The analogous statement holds for the vertical edges $v_{i,j}$. Thus the edge action is also associative, and commutes with $\partial_1$. Finally, a square $q_{i,j}$ is sent to another square (and not collapsed to $0$) by the action of an element if and only if neither of the coordinate directions are collapsed. In that case, the four vertices are easily checked to be sent to the four vertices of the square $q_{i+t,j-t}$, and otherwise one or both pairs of opposite boundary edges are mapped to the same image with opposite signs, and hence becomes $0$. This immediately yields associativity and that $\partial_2(q_{i,j} \cdot s) = \partial_2(q_{i,j}) \cdot s$, as desired. 

Notice now that the action on squares is diagonal-preserving or degenerate, i.e.\ that if $i+j=n$ then $q_{i,j} \cdot s$ is either $0$ or equal to a square $q_{i', j'}$ with $i'+j' = n$. Hence, we have a grading by $n \geq 0$ of $C_2(\cQ)$ as follows: let $C_{2}^{(n)}(\cQ) = \bigoplus_{i+j=n} \Z[q_{i,j}]$. Then $C_2(\cQ) = \bigoplus_{n \geq 0} C_{2}^{(n)}(\cQ)$, and every $C_2^{(n)}(\cQ)$ is a right $\Z\FF_1$-submodule of $C_2(\cQ)$. Set
\begin{equation}\label{Eq:L_n-decomp}
L_n := \partial_2(C_2^{(n)}(\cQ)) = \partial_2 \left(  \bigoplus_{i+j=n} \Z[q_{i,j}] \right).
\end{equation}
Notice now that since $\partial_2$ is an isomorphism onto $\cZ_1(\cQ)$, we see that every $L_n$ is a right $\Z\FF_1$-submodule of $\cZ_1(\cQ)$. This is, as we shall presently see, the core reason behind the failure of $\FP_2$ for free inverse monoids.

\subsection{Idempotents give many cycles}\label{Subsec:idempotents-give-cycles}

Let now $M$ be a finitely generated monoid (not necessarily inverse) and $\tau \colon M \to \FF_1$ any homomorphism. Then by restriction of scalars via $\tau$, the complex $\cQ$ becomes a right $\ZM$-module. Recall the (not necessarily free) resolution constructed for $M$ in \S\ref{Subsec:hom-setup}, with $P_0 = \ZM, P_1= (\ZM)^A$, and $K_A$ the second syzygy module. Let $\mathbf{0} = (0,0)$ be a base vertex of $\cQ$. We define a $\ZM$-linear map $\Phi_0 \colon \ZM \to C_0(\cQ)$ by for all $m \in M$ setting $\Phi_0(m) = \mathbf{0} \cdot m$. Note that $(\varepsilon \circ \Phi_0) (m) = \varepsilon(m)$ for all $m \in M$ and extending $\Z$-linearly. Next, for each $a \in A$ let $\eta_a$ be a directed edge path from $\mathbf{0}$ to $\mathbf{0} \cdot a$ using only positively oriented horizontal and vertical edges, i.e.\ $\eta_a$ is some path connecting $\mathbf{0}$ with its image under the action of $a$. Then define a map $\Phi_1 \colon (\ZM)^A \to C_1(\cQ)$ by mapping the basis element $(a)$ of $(\ZM)^A$ corresponding to $a \in A$ to $\eta_a$, and extending $\ZM$-linearly, i.e.\ $\Phi_1( (a)m ) = \eta_a \cdot m$ for all $a \in A$ and $m \in M$. Then for all basis elements $(a) \in (\ZM)^A$ we have
\[
(\Phi_0 \circ d_1)( (a) ) = \Phi_0 ( a - 1 ) = \mathbf{0}\cdot a - \mathbf{0} = \partial_1 (\eta_a) = (\partial_1 \circ \Phi_1) ((a)),
\]
so that $\Phi_0 d_1 = \partial_1 \Phi_1$. We now can construct a map $\Phi_2 \colon K_A \to C_2(\cQ)$ by a standard diagram chase. Indeed, let $w \in K_A = \ker(d_1)$. Then $(\Phi_0 \circ d_1)(w) = 0$, so $(\partial_1 \circ \Phi_1)(w) = 0$, so $\Phi_1(w) \in \ker(\partial_1) = \cZ_1(\cQ)$. Since $\cQ$ is contractible we have that $\partial_2$ restricts to an isomorphism $C_2(\cQ)\cong \cZ_1(\cQ)$, so in particular there is a unique preimage $q$ of $\Phi_1(w)$ in $C_2(\cQ)$. We then define $\Phi_2(w)$ to be this preimage, i.e.\ $\Phi_2(w) = q$. Note that $\Phi_2$ is, by construction, a $\ZM$-linear map, since $\partial_2$ is an isomorphism of $\ZM$-modules. We thus have the following diagram where all squares commute.
\[\begin{tikzcd}[column sep=large, row sep=large]
	0 & {K_A} & {(\ZM)^A} & \ZM & \Z & 0 \\
	0 & {C_2(\cQ)} & {C_1(\cQ)} & {C_0(\cQ)} & \Z & 0
	\arrow[from=1-1, to=1-2]
	\arrow[hook,from=1-2, to=1-3]
	\arrow["{\Phi_2}", from=1-2, to=2-2]
	\arrow["{d_1}", from=1-3, to=1-4]
	\arrow["{\Phi_1}", from=1-3, to=2-3]
	\arrow["\varepsilon", from=1-4, to=1-5]
	\arrow["{\Phi_0}", from=1-4, to=2-4]
	\arrow[from=1-5, to=1-6]
	\arrow["\id", from=1-5, to=2-5]
	\arrow[from=2-1, to=2-2]
	\arrow["{\partial_2}"', from=2-2, to=2-3]
	\arrow["{\partial_1}"', from=2-3, to=2-4]
	\arrow["\varepsilon"', from=2-4, to=2-5]
	\arrow[from=2-5, to=2-6]
\end{tikzcd}\]
Furthermore, we of course immediately get an exact sequence 
\begin{equation}
0 \to \ker(\Phi_2 \mid_{K_A}) \to K_A \to \Phi_2(K_A) \to 0.
\end{equation}
Hence, proving $K_A$ is not finitely generated reduces to showing that $\Phi_2(K_A) \subseteq C_2(\cQ)$ is not finitely generated. Recall from \S\ref{Subsec:grid-action} that $C_2(\cQ)$ is graded as $\bigoplus_{n \geq 0} C_2^{(n)}(\cQ)$, corresponding to the level $n=i+j$ of $q_{i,j}$. The following is now our key lemma.

\begin{lemma}\label{Lem:key-lemma}
Let $e, f \in E(M)$ be a pair of commuting idempotents, and suppose that $\tau(e)$ and $\tau(f)$ are incomparable. Let $n = \lambda(\tau(ef))$ be the level of the idempotent $\tau(ef)$. Then there is some $z \in K_A$ such that the support of $\Phi_2(z) \in C_2(\cQ)$ has a non-zero component in level $n-2$. 
\end{lemma}
\begin{proof}
For a word $w \equiv a_1 a_2 \cdots a_n \in A^\ast$ with $a_i \in A$, let $p_w = \sum_{i=1}^n (a_i) a_{i+1} \cdots a_n \in (\ZM)^A$. Then it is easy to see, by telescoping, that $d_1(p_w) = w-1$, and that if we let
\[
\eta_w := \Phi_1(p_w) = \sum_{i=1}^{n} \eta_{a_i} \cdot a_{i+1} \cdots a_n
\]
then $\eta_w$ is a directed path from $\mathbf{0}$ to $\mathbf{0} \cdot w$. 

Let now $e, f$ be as in the statement of the lemma. Assume, after interchanging $e$ and $f$ if necessary, that $\tau(e) = \varepsilon_{i,j}$ and $\tau(f) = \varepsilon_{k,\ell}$ with $i>k$ and $j < \ell$. Then $\tau(ef) = \varepsilon_{i,\ell}$. Write $e = a_1 \cdots a_n$ and $f = b_1 \cdots b_k$, where $a_\mu, b_\nu \in A$ for all $\mu, \nu$. Let $\eta_e = \Phi_1(p_e)$ and $\eta_f = \Phi_1(p_f)$ as above. Let
\[
z = p_f + p_e f - p_e - p_f e
\]
being, informally speaking, the element corresponding to going around the square obtained from having $ef = fe$. We claim that $z$ has the required property in the conclusion of the lemma. First, note that $z \in K_A$. Indeed, 
\[
d_1(z) = (f-1) + (e-1)f - (e-1) - (f-1) e = ef-fe = 0.
\]
We now have $\Phi_1(z) = \eta_f + \eta_e \cdot f - \eta_e - \eta_f \cdot e$. Let $\cR$ be the subcomplex of $\cQ$ with vertex set $\{ 0, \dots, i\} \times \{ 0, \dots, \ell \}$. The cycle $\Phi_1(z)$ is entirely supported in $\cR$, since $\eta_e$ and $\eta_f$ are directed paths using only positively oriented edges. The complex $\cR$ is contractible, so $\Phi_1(z)$ bounds a $2$-chain supported in $\cR$, and since $\partial_2 \colon C_2(\cQ) \to \cZ_1(\cQ)$ is injective, this $2$-chain must be the unique filling $\Phi_2(z)$. Hence $\Phi_2(z)$ is supported in $\cR$. Write
\[
\Phi_2(z) = \sum_{r,s \geq 0} c_{r,s}q_{r,s},
\]
where $c_{r,s}\in \Z$, and consider the top-right edge $v_{i,\ell-1} \colon (i,\ell-1) \to (i,\ell)$. It is easy to see that the path $\eta_f \cdot e$ contains $v_{i,\ell-1}$ exactly once, with coefficient $1$. On the other hand, since $i>k$ and $\ell>j$, the paths $\eta_e$ and $\eta_f$ do not contain $v_{i,\ell-1}$, and neither can the path $\eta_e \cdot f$. Hence $v_{i,\ell-1}$ has coefficient $1$ in $\eta_f \cdot e$, and hence $v_{i,\ell-1}$ has coefficient $-1$ in $\Phi_1(z)$, since $p_f e$ has coefficient $-1$ in $z$. Finally, since the only square $q_{r,s}$ in $\cR$ which contains the top right edge $v_{i,\ell-1}$ is $q_{i-1,\ell-1}$, it follows that $c_{i-1, \ell-1} = -1$. The level of the square $q_{i-1,\ell-1}$ is $(i-1) + (\ell - 1) = i + \ell -2 = \lambda(\tau(ef)) - 2 = n-2$, as desired.
\end{proof}

This lemma is sufficient to prove our main theorem.

\subsection{A criterion for failing $\FP_2$}

We can now present the main theorem of this article, which amounts to saying that any monoid $M$ which maps homomorphically to sufficiently large parts of the idempotent grid of $\FF_1$ is not of type $\FP_2$. 

\begin{theorem}\label{Thm:main-theorem}
Let $M$ be a finitely generated monoid, and let $\tau \colon M \to \FF_1$ be a homomorphism. Suppose that $M$ contains infinitely many pairs of commuting idempotents $e_i, f_i \in E(M)$ for $i = 1, 2, \dots$ such that $\tau(e_i)$ and $\tau(f_i)$ are incomparable, and such that $\tau(e_i)\neq\tau(e_j)$ for $i \neq j$. Then $M$ is not of type left- or type right-$\FP_2$. 
\end{theorem}
\begin{proof}
First, we notice that since all Green's $\mathscr{D}$-classes of $\FF_1$ are finite, for every $N \geq 0$ there are only finitely many idempotents in $\FF_1$ of level at most $N$, and hence the hypotheses of the theorem imply that the set of levels $\{ \lambda(\tau(e_i f_i)) \mid i = 1, 2, \dots \}$ is unbounded. By Lemma~\ref{Lem:key-lemma}, the unboundedness of these levels implies that the image $\Phi_2(K_A)$ of the syzygy module $K_A$ is not supported in finitely many grades inside $C_2(\cQ) = \bigoplus_{n \geq 0} C_2^{(n)}(\cQ)$. Hence $\Phi_2(K_A)$ cannot possibly be finitely generated; and hence neither can $K_A$. Thus $M$ is not of type right-$\FP_2$. All arguments performed up to this point have been left-right agnostic, once a handedness has been fixed; thus, \textit{mutatis mutandis}, we also conclude that $M$ is not of type left-$\FP_2$. 
\end{proof}

Of course, an immediate corollary is that if we hit \textit{all} of the idempotents of $\FF_1$, then $M$ is not of type $\FP_2$. In other words, we immediately conclude: 

\begin{corollary}\label{Cor:if-surjects-then-not-FP2}
Let $M$ be a finitely generated inverse monoid which surjects onto $\FF_1$. Then $M$ is not of type left- or right-$\FP_2$. 
\end{corollary}
\begin{proof}
Indeed, every idempotent in $\FF_1$ has an idempotent element in its pre-image, and these idempotents commute because $M$ is inverse. Thus Theorem~\ref{Thm:main-theorem} applies. 
\end{proof}

Thus we recover the theorem of Gray \& Steinberg mentioned in the introduction, avoiding their argument via normal forms. We also avoid their usage of a non-trivial result due to Pride \cite{Pride} that being of type $\FP_2$ is preserved under taking retractions. 

\begin{corollary}[{Gray \& Steinberg, 2021 \cite{GraySteinberg2021}}]\label{Cor:ff1-not-}
Let $\FF_r$ be the free inverse monoid of rank $r \geq 1$. Then $\FF_r$ is not of type left- or right-$\FP_2$. 
\end{corollary}

Note that the case $r=1$ of Corollary~\ref{Cor:ff1-not-} together with the aforementioned retraction result due to Pride \cite{Pride} together imply Corollary~\ref{Cor:if-surjects-then-not-FP2} by the universal property, which then implies Corollary~\ref{Cor:ff1-not-} for all ranks $r \geq 1$. We have thus avoided this detour by our direct argument. We finally also note that, as in \cite{GraySteinberg2021}, we thus also recover Schein's \cite{Schein1975} classical result that free inverse monoids of rank $r \geq 1$ are not finitely presented.

The main theorem is not limited to surjections, and we now apply it to answer the question of Cho \& Ru\v{s}kuc from the introduction. We recall the setting. In their article \cite{ChoRuskuc2025}, the authors proved that a finitely generated submonoid of $\FF_1$ is finitely presented if and only if it contains finitely many idempotents, and asked if $\FP_2$ is also equivalent to this property.  We give a positive answer to this question as follows:

\begin{corollary}\label{Cor:cho-ruskuc-question}
Let $M$ be a finitely generated submonoid of $\FF_1$. Then t.f.a.e.:
\begin{enumerate}
\item $M$ is finitely presented; 
\item $M$ is of type (left-) right-$\FP_2$; 
\item $M$ contains finitely many idempotents.
\end{enumerate}
\end{corollary}
\begin{proof}
The equivalence (1) $\iff$ (3) is \cite[Proposition~3.1+4.1]{ChoRuskuc2025}. The implication (1) $\implies$ (2) is always true (see \S\ref{Subsec:hom-setup}). We prove (2) $\implies$ (3) contrapositively. If $M$ contains infinitely many idempotents, then by \cite[Lemma~3.2 + Remark 3.3(3)]{ChoRuskuc2025} there exist infinitely many distinct pairs of \textit{incomparable} idempotents in $M$. Hence Theorem~\ref{Thm:main-theorem} applies with $\tau \colon M \hookrightarrow \FF_1$ being the inclusion map. 
\end{proof}

\bibliographystyle{amsalpha}
\bibliography{fp2-fim.bib}

\end{document}